\documentclass[11pt]{amsart}
\usepackage{amssymb}
\usepackage{rotating}
\usepackage{epsfig}
\usepackage{amsmath}
\usepackage{a4wide}

\newcommand{\R}{\ensuremath{\mathbb{R}}}

\newcommand{\X}{\ensuremath{\mathbb{X}}}

\newcommand{\F}{\ensuremath{{\mathcal F}}}
\newcommand{\T}{\ensuremath{{\mathcal T}}}

\DeclareMathOperator{\freq}{freq}
\DeclareMathOperator{\vol}{vol}

\newtheorem{thm}{Theorem}[section]
\newtheorem{lemma}[thm]{Lemma}

\newtheorem{defi}[thm]{Definition}

\begin{document}

\title{About Substitution Tilings with Statistical Circular Symmetry}

\author{D. Frettl\"oh}
\address{   Fakult\"at f\"ur Mathematik, Bielefeld,  Germany} 
\email{dirk.frettloeh@math.uni-bielefeld.de}
\urladdr{http://www.math.uni-bielefeld.de/baake/frettloe}

\maketitle

\begin{abstract}
Two results about equidistribution of tile orientations in
primitive substitution tilings are stated, one for finitely many,
one for infinitely many orientations. Furthermore, consequences for
the associated diffraction spectra and the dynamical systems are
discussed.    
\end{abstract}

\section{Substitution tilings}
Several important models of quasicrystals are provided by
(decorations of) nonperiodic substitution tilings. In this paper, only
tilings of the plane are considered. In larger dimensions some aspects
become rather complicated, in particular, generalisations of Theorem
\ref{unifdistr}. A substitution (or tile-substitution, but we stick to
the shorter term here) is just a rule how to enlarge and dissect a
collection of tiles $P_1, \ldots , P_m$ into copies of $P_1, \ldots,
P_m$, as indicated in Figure \ref{imbal} (left) or Figure
\ref{substbsp}. Iterating this rule yields tilings of larger and
larger portions of the plane. However, the precise description of a
substitution tiling in $\R^2$ is usually given as follows.  

Let $P_1, \ldots, P_m$ be compact subsets of $\R^2$, such that each
$P_i$ equals the closure of its interior. These are the 
{\em  prototiles}, the pieces our tiling is built of. We refer to 
sets which are congruent to some prototile simply as {\em tiles}. 
(Sometimes two prototiles are allowed to be congruent, compare Figure
\ref{imbal}. Then we equip each one of them with an additional
attribute (colour, decoration) to distinct them. Whenever we speak of
'congruent' tiles in the sequel, this means 'congruent with respect to
possible decorations'.) A {\em tiling} (with prototiles $P_1, \ldots,
P_m$) is a collection of tiles (each one congruent to some $P_i$)
which covers $\R^2$ such that every $x \in \R^2$ is 
contained in the interior of at most one tile. Note, that each such
tiling $\T$ can be described by placements of the prototiles: 
\[ \T = \{ R_{\alpha_1} P_{i_1} + t_1, R_{\alpha_2} P_{i_2} + t_2, \ldots \}, \]
where $R_{\alpha_i}$ denotes a rotation through $\alpha_i$, and 
$t_i \in \R^2$ are translation vectors. For later purposes, we define
the {\em orientation} $\phi(T)$ of a tile: If $T = R_{\alpha} P_i +
t$, $\alpha \in [0, 2\pi[$, then $\phi(T)=\alpha$.  
 
Now, let $\lambda > 1$ and $\sigma$ be a map which maps each
prototile $P_i$ to a non-overlapping collection of tiles, each one
congruent to some prototile $P_i$. Furthermore, we require the
substitution $\sigma$ to be  {\em  selfsimilar} in the sequel. That
is, $\lambda P_i =  \bigcup_{T \in \sigma(P_i)} T$ for $1 \le i \le m$. 
By setting $\sigma(R P_i + t) = R \sigma(P_i) + \lambda t$, the
substitution $\sigma$ extends naturally to any tile congruent to some
prototile. In a similar way, by $\sigma(\{T_1, T_2, \ldots \}) =
\sigma(T_1) \cup \sigma (T_2) \cup \ldots$, the substitution $\sigma$
extends to all (finite or infinite) collections of tiles. In
particular, one can iterate $\sigma$ on the prototiles to obtain the
$k$-th order {\em supertiles} $\sigma^k(P_i)$. 

\begin{defi} \label{defsubst}
Let $\sigma$ be a tile-substitution with prototiles $P_1, \ldots,
P_m$. A tiling $\T$ is called {\em substitution tiling} (with
substitution $\sigma$) if for each finite subset $\F \subset \T$
there are $i,k$ such that a copy of $\F$ is contained in some
supertile $\sigma^k(T_i)$. The family of all substitution tilings with
substitution $\sigma$ is denoted by $\X_{\sigma}$. The matrix
$M_{\sigma}$ with $(M_{\sigma})_{ij} = \# \{ T \in \sigma(P_j) \, : \,
T = R  P_i + t \}$ is called {\em substitution matrix}. A substitution
is {\em primitive}, if there is $k \ge 1$ such that $M^k_{\sigma}$ is
strictly positive.    
\end{defi}

Here, $\# A$ denotes the cardinality of the set $A$. 
This definition of a substitution tiling fits well into several
contexts, for instance, to the {\em hull} of a tiling, see Section
\ref{dynsys}.

\section{Statistical circular symmetry} \label{sec2}

Most of the well-known substitution tilings have the property that 
all prototiles occur in only finitely many orientations. For instance,
the Penrose dart and kite tiling has just two prototiles, a convex and a 
non-convex quadrangle. Each one occurs in one of 10 different 
orientations throughout the tiling. 
There are also substitution tilings where the tiles occur in infinitely
many orientations. The most prominent examples are certainly the
pinwheel tilings \cite{rad}. But there are many other examples,
compare \cite{f}, \cite{fh}. Two examples are shown in Figure
\ref{substbsp}.  

It was known for the pinwheel tiling that there are not only
infinitely many orientations of the tiles, but that they are
furthermore equidistributed \cite{rad}, \cite{mps}. Intuitively, this
means that each orientation occurs with the same frequency throughout
the tiling. In order to make this precise, we need the following
definition. Recall that a sequence $(a_j)_{j \ge 0}$ is 
{\em  equidistributed} in $[0, 2\pi[$, if for all $0 \le x < y < 2\pi$
holds: 
\[ \lim_{n \to \infty} \frac{1}{n} \sum_{j=1}^{n} 
  1_{[x, y]}(a_j)  = \frac{x - y}{2 \pi}. \]
Since the above sum is not absolutely convergent,
it depends strongly on the ordering of the sequence. Thus we need to
order the tiles in the tiling in some not-too-weird way. This is done
in the following definition. 

\begin{defi}
Let $\T = \{ T_1, T_2, \ldots \}$ be a primitive
substitution tiling. Let the numbering be such that the sequence
$(T_j)^{}_{j \ge 1}$ satisfies the following property: for all 
$n \ge 1$, there is some $\ell \ge n$ such that the patch 
$\{ T_1 \ldots , T_{\ell} \}$ is congruent to some supertile
$\sigma^k(P_i)$ for some prototile $P_i$, $k \ge 1$. The tiling
$\T^{}_{\sigma}$ has {\em statistical circular symmetry}, if, for all
$0 \le x < y < 2 \pi$, one has: 
\[ \lim_{r \to \infty} \frac{1}{n} \sum_{j=1}^{n} 
  1_{[x, y]}(\phi(T_j))  = \frac{x - y}{2 \pi}. \] 
\end{defi}
The ordering in the definition is compatible with other natural
notions, for instance, ordering the tiles with respect to their distance
to the origin. The following result was obtained in \cite{f}. Only while
preparing the present text, the author realized that a similar result
already appeared in \cite{rad}; see also the comments in Section
\ref{dynsys}.  

\begin{thm} \label{unifdistr}
Let $\T$ be a primitive substitution tiling in $\R^2$, where
copies of some prototile occur in infinitely many orientations. 
Then $\T$ is of statistical circular symmetry. 
\end{thm}

\begin{figure}
\begin{center}
\epsfig{file=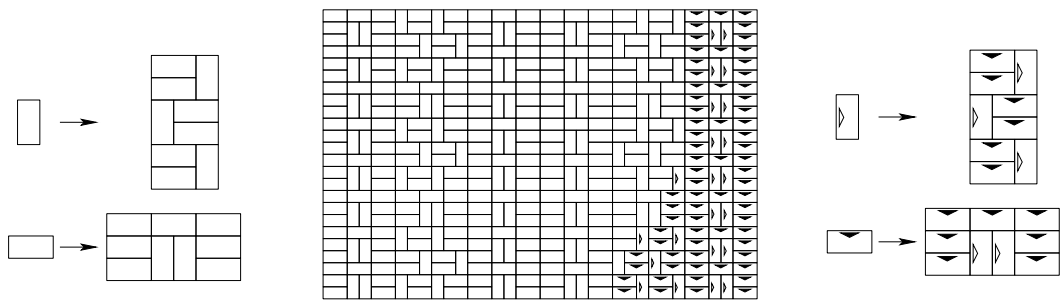, width=140mm}
\end{center}
\caption{\label{imbal} A substitution tiling where horizontal rectangles
  are more frequent in the tiling than vertical ones. By Theorem
  \ref{finor}, this can be achieved only when the two types of
  rectangles are substituted differently from each other.}
\end{figure}

In the next section this theorem is used to show that any such tiling
has circular diffraction spectrum. Now we show an analogue of the
theorem above for tilings with tiles in finitely many orientations. 
One can ask whether there are substitution tilings where tiles occur
more frequent in one direction than the other. This is possible, see
Figure \ref{imbal} for an example. But this is achieved by
substituting the horizontal rectangle differently from the vertical
rectangle. In fact we have two prototiles in this example, each one
occurring in one direction only. This is illustrated in the right part
of Figure \ref{imbal} by assigning different decorations to the
tiles. The following theorem states that 
each prototile $P_i$ occurs in each of its orientations with the same
frequency in a primitive substitution tiling. In order to make this
precise, let $\freq(P_i,\alpha)$ denote the frequency of tiles
$R_{\alpha} P_i +t$ in a tiling $\T$. That is, 
\begin{equation}\label{freq}
 \freq(P_i, \alpha, \T) = \lim_{r \to \infty} \frac{\# \{T \in \T \cap
  B_r \; : \; \exists t \; \; T = R_{\alpha} P_i + t \} }{ \# \{ T \in
  \T \cap   B_r \} },
\end{equation}
where $B_r$ denotes the closed ball of radius $r$ around the origin.
For primitive substitution tilings, this frequency exists and is
independent of the choice of the origin. This can be seen as follows: 
Up to now, we identified types of tiles with respect to congruence
congruence classes. In a tiling with finitely many orientations, one
can as well identify types of tiles with respect to translation classes.  
Usually, this results in a larger number of prototiles (e.g., the
Penrose dart and kite tilings have two prototiles up to congruence,
but 20 prototiles up to translation). The frequency of each prototile
in a primitive substitution tiling can be obtained from the
substitution matrix $M_{\sigma}$. This result is folklore, see for
instance \cite{fogg}, \cite{f1}: The normed Perron-Frobenius
eigenvector of $M_{\sigma}$ contains the relative frequencies of the
prototiles. In particular, in all tilings in $\X_{\sigma}$, these
frequencies are the same. Therefore, for any
given primitive substitution, we can drop the $\T$ in the definition
of the frequency in this case and write just $\freq(P_i, \alpha)$. 

\begin{thm} \label{finor}
Let $\sigma$ be a primitive substitution with prototiles $\{P_1,
\ldots, P_m \}$. If $\freq(P_i,\alpha) \ne 0$, $\freq(P_i,\beta) \ne
0$ then $\freq(P_i,\alpha) = \freq(P_i,\beta)$. 
\end{thm}
\begin{proof}
If $\freq(P_i,\alpha) \ne 0$,
then the prototiles occur in finitely many orientations only, by
primitivity of the substitution and Theorem \ref{unifdistr}. Therefore
we obtain finitely many prototiles with respect to translations: $(P_1,0), 
(P_1, \alpha_{1,1}),  (P_1, \alpha_{1,2}),  \ldots, (P_1, \alpha_{1,k});
(P_2,0), (P_2, \alpha_{2,1}), \ldots \ldots (P_m, \alpha_{m, \ell} )$. 

Consider a tiling $\T \in \X_{\sigma}$. By the remarks preceding the
theorem, $\freq(P_i,0)$ is well defined, and its value is the
same in the tiling $\T$ and in the tiling $R_{\alpha_{i,j}} \T \in
\X_{\sigma}$. And trivially, $\freq(P_i, 0)$ in $\T$ equals
$\freq(P_i, \alpha_{i,j} )$ in $R_{\alpha_{i,j}} \T$. Therefore,
$\freq(P_i, 0 ) = \freq(P_i, \alpha_{i,j} )$, and the claim follows.  
\end{proof}

\begin{figure}
\begin{center}
\epsfig{file=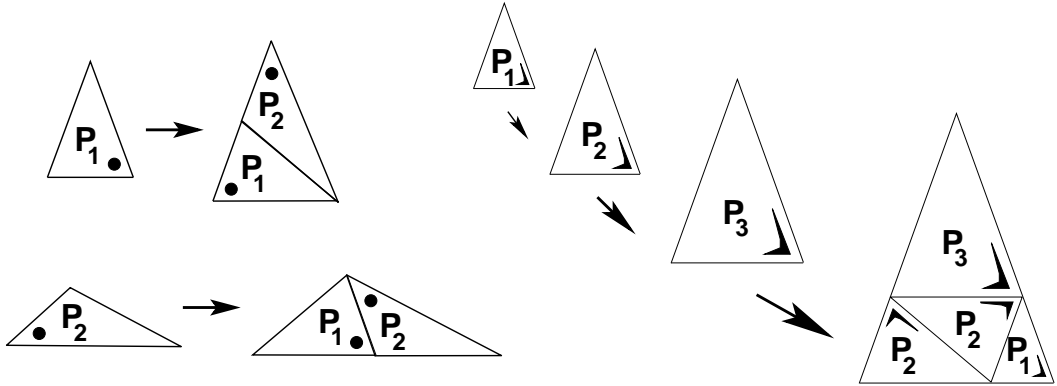, width=140mm}
\end{center}
\caption{\label{substbsp} Two substitutions yielding tiling with
  statistical circular symmetry. The left one is due to Danzer and
  Goodman-Strauss \cite{cgs}, the right one is from \cite{f}.}
\end{figure}

\section{Diffraction}

There is a wealth of literature devoted to the mathematical
description of the diffraction spectrum of nonperiodic structures, see
for instance \cite{bl}, \cite{hof}, \cite{l}, \cite{schl}, or the
introductory text \cite{baa} and references therein. Usually one
starts with a discrete point set $\Lambda \subset \R^2$ rather than a
tiling. In the present context of tilings, this can be achieved by
decorating each prototile with one or more points ('atomic
decorations'). 

Let $\Lambda$ be such a point set, obtained from a tiling. 
The diffraction spectrum of $\Lambda$ is the Fourier transform of the
autocorrelation of $\Lambda$.
The autocorrelation (also called Patterson function) of $\Lambda$ is 
\begin{equation} \label{autocor}
 \gamma = \gamma_{\Lambda} =
  \lim_{r \to \infty} \frac{1}{\vol(B_r)} \sum_{ x,y \in
  \Lambda \cap B_r} \delta_{x-y}, 
\end{equation}
if this limit exists. Here and in the following we are working in the
framework of tempered distributions \cite{rud}, \cite{schw}. Thus the 
limit above is to be understood as a vague limit. If the limit does
not exist, there is a subsequence converging to
an autocorrelation which is known to be translation bounded (since we
started with a translation bounded measure $\delta_{\Lambda}$) and
positive definite. Thus $\gamma$ can be regarded as a translation 
bounded positive definite measure. Now, the Fourier transform
$\widehat{\gamma}$ of $\gamma$ is the {\em diffraction measure}, and
its support is the diffraction spectrum of $\Lambda$. We are
interested in the nature of $\widehat{\gamma}$. By Lebesgue's
decomposition theorem, $\widehat{\gamma}$ has a unique decomposition
\[ \widehat{\gamma} = \widehat{\gamma}_{pp} + \widehat{\gamma}_{sc} +
\widehat{\gamma}_{ac} \] 
with respect to the Lebesgue measure in $\R^2$. The pure point part 
$\widehat{\gamma}_{pp}$ is a countable sum of Dirac measures. The 
absolutely continuous part $\widehat{\gamma}_{ac}$ is a measure with
locally integrable density function supported on a set of positive
Lebesgue measure. The singular continuous $\widehat{\gamma}_{sc}$ part
assigns zero to all finite sets, but is supported on a set with
Lebesgue measure zero. The (idealized) diffraction spectrum of a
quasicrystal is pure point, the singular parts vanish completely. 
 
The following results show that substitution tilings with statistical
circular symmetry have a diffraction measure which is of (perfect)
circular symmetry. Consequently, the pure point part of the
diffraction is trivial. 

\begin{thm} \label{thmautocor}
Let $\T = \{ R_{\alpha_1} P_{i_1} + t_1, R_{\alpha_2} P_{i_2} + t_2,
\ldots \}$ be a primitive substitution tiling of statistical circular
symmetry, with prototiles $P_1, \ldots, P_m$. Choose 'control points'
$x_i \in P_i$, and define $\Lambda = \{ R_{\alpha_1} x_{i_1} + t_1,
R_{\alpha_2} x_{i_2} + t_2, \ldots \}$. Then the autocorrelation
$\gamma_{\Lambda}$ of $\Lambda$ is circular symmetric.
\end{thm}

The subtlety of this theorem is that the autocorrelation is not of
statistical circular symmetry only (this is easy to see) but of
perfect circular symmetry.  The rotated measure $R.\mu$ is defined
by $R.\mu(A) = \mu(R^{-1} A)$ for all measurable sets $A$. Then $\mu$
is of circular symmetry if $R. \mu = \mu$ for all $R$.   

\begin{proof}
Since $\T$ is of statistical circular symmetry, $\T$ contains tiles
$T_j$ congruent to some prototile $P_i$ for (countably) infinitely many 
distinct angles $\alpha_j = \phi(T_j)$. 
By equidistribution, the frequency of tiles in one certain orientation
is 0. But, again by the equidistribution of $(\alpha_j)_{j \ge 0}$,
the frequency of tiles in a certain range of orientations --- say,  
$[ \beta, \gamma [ \subseteq [0, 2 \pi[$ --- equals the frequency of
tiles in orientations $[ \beta+c , \gamma+c [ \mod 2 \pi$ for all $c$.
The same is obviously true for supertiles in $\T$. 

The autocorrelation is defined via difference vectors $x-y$. 
We deduce the circular symmetry of the autocorrelation by assigning
these vectors $x-y$ to supertiles in $\T$. Each such
vector arises from a constellation of two tiles (usually not adjacent). 
Consider only vectors $x-y$ of fixed a length $r>0$. Let $k \ge 1$. 
Then either $x$ and $y$ are contained in the same $k$-th order
supertile, or in two different $k$-th order supertiles. With growing
$k$, the ratio of the latter kind to the former kind tends to zero.
Thus it suffices to regard the former kind only. Each one arises from
a $k$-th order supertile, these are equidistributed, hence
equidistribution of orientations of difference vectors $x-y$ of radius
$r$ follows. In other words, the frequency of difference vectors $x-y$
of length $r$ with orientation $\varphi \in [ \beta, \gamma [
\subseteq [0, 2 \pi[$ equals the frequency of difference vectors of
length $r$ in orientations $[ \beta+c , \gamma+c [ \mod 2 \pi$ for all
$c$. It follows that  $\gamma_{\Lambda} = R.\gamma_{\Lambda}$ for all
rotations $R$.  
\end{proof}

Now we can apply the following standard result, which can be found for
instance in \cite{bf}. 

\begin{lemma} \label{taurlemma}
Let $R$ be an orthogonal map, $\mu$ a measure, and let the measure 
$R.\mu$ be given by $R. \mu (A) = \mu(R^{-1}A)$. Then
$  R.\widehat{\mu} = \widehat{R. \mu}\, $. 
\end{lemma}

Since $R. \gamma_{\Lambda} = \gamma_{\Lambda}$, with $\Lambda$ as in
\ref{thmautocor}, it follows $R. \widehat{\gamma_{\Lambda}} =
  \widehat{R. \gamma_{\Lambda}} =  \widehat{\gamma_{\Lambda}}$ for all
rotations $R$. Altogether we have obtained the following result.

\begin{thm} 
Let $\Lambda$ be as in Theorem \ref{thmautocor} arising from a tiling
with statistical circular symmetry. Then the diffraction of $\Lambda$
is circular symmetric. \hspace{\fill} $\Box$
\end{thm}

An immediate consequence is that the pure point part of the
diffraction spectrum of a tiling with statistical circular symmetry is
at most one peak at the origin. A simple argument shows that this peak
always exists, compare \cite{hof}, \cite{mps}. These results are
relevant for physics because such tilings can serve as a simple model 
for powder diffraction \cite{bfg}.

\section{Tiling dynamical systems} \label{dynsys}

An important tool in investigating properties of nonperiodic tilings
is the fact that such a tiling gives rise to a dynamical system,
compare \cite{radw},\cite{sol},\cite{schl},\cite{bl}, \cite{rob}. From
the properties of the dynamical system one can deduce properties of the
tiling. For brevity, we describe the metrical dynamical system
corresponding to
\cite{rad}. In the context of this paper, the topology defined by
this metric is the same as the more general local rubber topology
in \cite{bl}, as well as the topology in \cite{radw} and the local
topology in \cite{schl}. Let $d$ be the following metric, measuring
whether two tilings are 'close' to each other. 
\[ d(\T,\T') =  \min \big\{ \frac{1}{\sqrt{2}} , \inf_{\varepsilon >0}  \{
\varepsilon \, : \, B_{1/\varepsilon} \cap (\T + s) =
B_{1/\varepsilon} \cap (R_{\alpha} \T' + t), \, \|s\|, \|t\| \le
\frac{\varepsilon}{2}, \, |\alpha| \le \varepsilon \} \big\}  \] 
If $d(\T,\T')$ is small, then $\T$ and $\T'$ agree on a large ball
around the origin, after a small rotation followed by a small
translation. Now, the {\em hull}\, $\X_{\T}$ of the tiling $\T$ is the
closure of $\{ \T + t \, : \, t \in \R^2 \}$ with respect to $d$. It
is a well-known fact that the hull of a primitive substitution tiling
with substitution $\sigma$  is the family $\X_{\sigma}$ of all
substitution tilings (compare Def.\ \ref{defsubst}). Therefore, we
denote the  hull of $\T$ by $\X_{\sigma}$, if it is clear which
substitution $\sigma$ belongs to $\T$.  

A substitution $\sigma$ gives rise to a dynamical system
$(\X_{\sigma}, G)$, where either $G = \R^2$ (all translations in
the plane), or $G = E(2)$ (all direct Euclidean motions, i.e., all
maps $x \mapsto R x +t$, $R$ a rotation, $t$ a translation vector). 
The study of such dynamical systems yields deep insight into the
diffraction properties of the tilings in $\X_{\sigma}$. A milestone in
the mathematical diffraction theory is the result that the diffraction
spectrum is pure point if and only if the dynamical spectrum is pure
point \cite{hof}, \cite{sol}, \cite{schl}, \cite{lms}, \cite{bl}.  
Other important results are collected in the following theorem. 
It summarises contributions of several people to the subject over
the last two decades. For brevity, it is formulated for the special
case of plane tilings only. 

\begin{thm}
Let $\sigma$ be a primitive substitution. 
\begin{enumerate}
\item $\X_{\sigma}$ has {\em  uniform cluster frequencies}, that is, each
  patch occurs with a well defined frequency throughout the tilings in
  $\X_{\sigma}$. Here, frequency is defined analogously to \eqref{freq}.   
\item If the tiles in tilings in $\X_{\sigma}$ show finitely many
  orientations, then $(\X_{\sigma}, \R^2)$ is uniquely ergodic.
\item If the tiles in tilings in $\X_{\sigma}$ show infinitely
  many orientations, then $(\X_{\sigma}, G)$ is uniquely ergodic, where 
  $G \in \{ \R^2, E(2) \}$. 
\item All tilings in $\X_{\sigma}$ have the same autocorrelation, thus
  the same diffraction spectrum.
\end{enumerate}
\end{thm}

The first part is a simple consequence of Def.\ \ref{defsubst}, compare
the discussion in Section \ref{sec2}, or \cite{sol} for 'self-affine'
tilings. (Each $\X_{\sigma}$ as above contains a self-affine tiling,
and by primitivity, all tilings in $\X_{\sigma}$ have the same cluster
frequencies.) However, for tilings with statistical circular symmetry,
these frequencies are zero. But see \cite{f} for a similar uniformity
result in this case. A detailled proof of 
the second part can be found in \cite{sol} or \cite{lms}. For tilings
with tiles in infinitely many orientations, part three is stated in
\cite{rad} for $G = E(2)$. For $G = \R^2$, it can hopefully be shown
via Theorem \ref{unifdistr}, adapting methods from \cite{lms}. We aim
to give a precise and detailled proof in the future. The last part of
the theorem is an immediate consequence from part one and the
definition of the autocorrelation.

\section*{Acknowledgements}

It is a pleasure to thank Christoph Richard for many helpful
comments. This work was supported by the German Research Council
(DFG), within the CRC 701.

\end{document}